\documentclass[a4paper,12pt,oneside,headsepline]{scrartcl}

\usepackage{etoolbox}
\usepackage{ifdraft}

\usepackage{luainputenc}
\usepackage[T1]{fontenc}

\usepackage{lmodern}
\usepackage{fourier}

\usepackage{amssymb, amsfonts}

\usepackage{mathrsfs} 
\usepackage{dsfont}
\usepackage[usenames,dvipsnames]{xcolor}
\usepackage{lettrine}
\usepackage{url}

\usepackage{stmaryrd}

\usepackage[english]{babel}

\usepackage[draft=false]{scrlayer-scrpage}

\usepackage{wrapfig}
\usepackage{footmisc}
\usepackage{needspace}

\usepackage{amsmath}
\usepackage{nccmath}
\usepackage[thmmarks, amsmath, amsthm]{ntheorem}

\providebool{nobiblatex}
\boolfalse{nobiblatex}
\ifbool{nobiblatex}{%
}{%
  \usepackage[style=numeric-comp,giveninits,url=false,sortcites=true,maxbibnames=99]{biblatex}
  \bibliography{bibliography.bib}
}

\usepackage[pdftex,final,
            pdfborder={0 0 0}, colorlinks=true,
            linkcolor=BrickRed, citecolor=ForestGreen, urlcolor=RoyalBlue]{hyperref}

\ifdraft{%
\usepackage[textsize=scriptsize,colorinlistoftodos]{todonotes}
\usepackage{lineno}
}{}

\usepackage{booktabs}
\usepackage[inline]{enumitem}
\usepackage{float}
\usepackage{graphicx}
\usepackage{multicol}
\usepackage{scrtime}
\usepackage{tikz}
\usetikzlibrary{matrix,arrows,positioning}

\ifdraft{\date{\today}}{\date{}}

\ifdraft{%
\linenumbers
\newcommand*\patchAmsMathEnvironmentForLineno[1]{%
  \expandafter\let\csname old#1\expandafter\endcsname\csname #1\endcsname
  \expandafter\let\csname oldend#1\expandafter\endcsname\csname end#1\endcsname
  \renewenvironment{#1}%
     {\linenomath\csname old#1\endcsname}%
     {\csname oldend#1\endcsname\endlinenomath}}%
\newcommand*\patchBothAmsMathEnvironmentsForLineno[1]{%
  \patchAmsMathEnvironmentForLineno{#1}%
  \patchAmsMathEnvironmentForLineno{#1*}}%
\AtBeginDocument{%
\patchBothAmsMathEnvironmentsForLineno{equation}%
\patchBothAmsMathEnvironmentsForLineno{align}%
\patchBothAmsMathEnvironmentsForLineno{flalign}%
\patchBothAmsMathEnvironmentsForLineno{alignat}%
\patchBothAmsMathEnvironmentsForLineno{gather}%
\patchBothAmsMathEnvironmentsForLineno{multline}%
}}

\KOMAoptions{DIV=10}

\clearscrheadings
\automark[section]{subsection}

\renewcommand{\subsectionmark}[1]{}

\lohead{{\small\normalfont \headertitle}}
\rohead{{\small \headerauthors}}

\RedeclareSectionCommand[%
  font=\Large\sffamily\bfseries,%
  beforeskip=1\baselineskip,%
  afterskip=0.5\baselineskip,%
  indent=0em%
  ]{section}

\RedeclareSectionCommands[%
  font=\normalfont\bfseries,%
  afterskip=-1em%
  ]{subsection,subsubsection}

\RedeclareSectionCommands[%
  font=\normalfont\itshape,%
  afterskip=-1em,%
  indent=0pt,
  ]{paragraph}

\ifbool{nobiblatex}{}{%
  \AtEveryBibitem{\clearfield{doi}}
  \AtEveryBibitem{\clearfield{isbn}}
  \AtEveryBibitem{\clearfield{issn}}
  \AtEveryBibitem{\clearfield{pages}}
  \AtEveryBibitem{\clearlist{language}}

  \setlength\bibitemsep{1pt}
  
  \renewbibmacro*{in:}{}
  \DeclareFieldFormat
    [article,inbook,incollection,inproceedings,patent,thesis,unpublished]
    {title}{#1}
}

\newenvironment{enumeratearabic*}{
\begin{enumerate*}[label=(\arabic*)] 
}{
\end{enumerate*}
}

\newenvironment{enumerateroman*}{
\begin{enumerate*}[label=(\roman*)] 
}{
\end{enumerate*}
}

\numberwithin{equation}{section}

\theoremnumbering{arabic}
\newtheorem{theoremcounter}{theoremcounter}[section]
\theoremnumbering{Alph}
\newtheorem{maintheoremcounter}{maintheoremcounter}

\theoremstyle{plain}

\newtheorem{corollary}[theoremcounter]{Corollary}

\newtheorem{proposition}[theoremcounter]{Proposition}
\newtheorem{theorem}[theoremcounter]{Theorem}

\theoremstyle{plain}

\newtheorem{maintheorem}[maintheoremcounter]{Theorem}

\theoremstyle{definition}

\theoremstyle{remark}

\newtheorem{remark}[theoremcounter]{Remark}

\theoremstyle{nonumberremark}

\newcommand{\tx}{\ensuremath{\text}}

\newcommand{\nbd}{\nobreakdash-\hspace{0pt}}

\newcommand{\tbf}{\bfseries}

\newcommand{\bbA}{\ensuremath{\mathbb{A}}}

\newcommand{\bbone}{\ensuremath{\mathds{1}}}

\newcommand{\cA}{\ensuremath{\mathcal{A}}}

\newcommand{\cD}{\ensuremath{\mathcal{D}}}

\newcommand{\frake}{\ensuremath{\mathfrak{e}}}

\newcommand{\frakg}{\ensuremath{\mathfrak{g}}}

\newcommand{\frakz}{\ensuremath{\mathfrak{z}}}

\newcommand{\frakU}{\ensuremath{\mathfrak{U}}}

\newcommand{\rmf}{\ensuremath{\mathrm{f}}}

\newcommand{\rmL}{\ensuremath{\mathrm{L}}}
\newcommand{\rmM}{\ensuremath{\mathrm{M}}}

\newcommand{\rmR}{\ensuremath{\mathrm{R}}}

\newcommand{\td}{\tilde}
\newcommand{\wtd}{\widetilde}
\newcommand{\ov}{\overline}

\let\rightarroworig\rightarrow
\renewcommand{\rightarrow}
  {\DOTSB\protect\relbar\mspace{-9.7mu}\rightarroworig}

\renewcommand{\twoheadrightarrow}
  {\DOTSB\protect\rightarroworig\mspace{-15mu}\rightarroworig}

\let\leftarroworig\leftarrow
\renewcommand{\leftarrow}
  {\DOTSB\protect\leftarroworig\mspace{-9.7mu}\relbar}

\newcommand{\longhookrightarrow}
  {\DOTSB\protect\lhook\joinrel\mspace{-0.1mu}\relbar\joinrel\mspace{-0.2mu}\relbar\mspace{-11.7mu}\rightarroworig}
\newcommand{\longtwoheadrightarrow}
  {\DOTSB\protect\relbar\joinrel\rightarroworig\mspace{-15mu}\rightarroworig}

\newcommand{\ra}{\ensuremath{\rightarrow}}

\newcommand{\thra}{\ensuremath{\twoheadrightarrow}}

\newcommand{\lhra}{\ensuremath{\longhookrightarrow}}
\newcommand{\lthra}{\ensuremath{\longtwoheadrightarrow}}

\newcommand{\mto}{\ensuremath{\mapsto}}
\newcommand{\lmto}{\ensuremath{\longmapsto}}

\renewcommand{\Re}{\ensuremath{\mathrm{Re}}}

\renewcommand{\pmod}[1]{\ensuremath{\;(\mathrm{mod}\, #1)}}

\newcommand{\sgn}{\ensuremath{\mathrm{sgn}}}

\newenvironment{psmatrix}{\left(\begin{smallmatrix}}{\end{smallmatrix}\right)}

\newcommand{\lspan}{\ensuremath{\mathop{\mathrm{span}}}}

\newcommand{\ZZ}{\ensuremath{\mathbb{Z}}}
\newcommand{\QQ}{\ensuremath{\mathbb{Q}}}
\newcommand{\RR}{\ensuremath{\mathbb{R}}}
\newcommand{\CC}{\ensuremath{\mathbb{C}}}

\newcommand{\GL}[1]{\ensuremath{\mathrm{GL}_{#1}}}

\newcommand{\PGL}[1]{\ensuremath{\mathrm{PGL}_{#1}}}
\newcommand{\SL}[1]{\ensuremath{\mathrm{SL}_{#1}}}
\newcommand{\SO}[1]{\ensuremath{\mathrm{SO}_{#1}}}

\newcommand{\sym}{\ensuremath{\mathrm{sym}}}

\newcommand{\HS}{\mathbb{H}}

\newcommand{\ahol}{\mathrm{ahol}}

\newcommand{\fraksl}[1]{\mathfrak{sl}_{#1}}

\newcommand{\Ind}{\mathrm{Ind}}
\newcommand{\Res}{\mathrm{Res}}
\newcommand{\Ext}{\mathrm{Ext}}

\newcommand{\sfd}{m}

\newcommand{\ga}{\gamma}
\newcommand{\Ga}{\Gamma}

\newcommand{\Ehol}{E^{\mathrm{hol}}}
\newcommand{\Evec}{E^{\mathrm{vec}}}
\newcommand{\tdEvec}{\td{E}^{\mathrm{vec}}}

\newcommand{\bbAf}{\bbA_\rmf}

\newcommand{\quantsep}{\ {}.{}\ }

\newcommand{\headertitle}{{%
  The Tensor Product Functor and Weight-$2$ Eisenstein Series
}}
\newcommand{\headerauthors}{%
  M.~Raum%
}
\title{%
The Bernstein--Gelfand\\Tensor Product Functor and\\
the Weight-$2$ Eisenstein Series
}
\author{%
Martin Raum%
\thanks{The author was partially supported by Vetenskapsr\aa det Grant~2019-03551.}%
}
\ifdraft{%
\date{\today\ at\ \thistime}
}

\begin{document}

\thispagestyle{scrplain}
\begingroup
\deffootnote[1em]{1.5em}{1em}{\thefootnotemark}
\maketitle
\endgroup


\begin{abstract}
\small
\noindent
{\tbf Abstract:}
The Bernstein--Gelfand tensor product functors are endofunctors of the category of Harish-Chandra modules provided by tensor products with finite dimensional modules. We provide an automorphic analogue of these tensor product functors, implemented by vector-valued automorphic representations that are trivial at all finite places. They naturally explain the role of vector-valued modular forms in recent work by Bringmann--Kudla on Harish-Chandra modules associated with harmonic weak Maa\ss{} forms. We give a detailed account of the image~$\sym^1 \otimes \varpi(E_2)$ of the automorphic representation~$\varpi(E_2)$ generated by the Eisenstein series of weight~$2$ under one of those tensor product functors. This builds upon work by Roy--Schmidt--Yi, who recently determined the structure of~$\varpi(E_2)$. They found that~$\varpi(E_2)$ does not decompose as a restricted tensor product over all places of~$\QQ$, while we discover that~$\sym^1 \otimes \varpi(E_2)$ has a direct summand that does. This summand corresponds to a holomorphic and modular, vector-valued analogue of~$E_2$. The complement in~$\sym^1 \otimes \varpi(E_2)$ arises from one of the vector-valued examples in the work of Bringmann--Kudla. Our approach allows us to determine its structure at the finite places.
\\[.3\baselineskip]
\noindent
\textsf{\textbf{%
  mock modular forms%
}}%
\noindent
\ {\tiny$\blacksquare$}\ %
\textsf{\textbf{%
  harmonic weak Maa\ss{} forms%
}}%
\noindent
\ {\tiny$\blacksquare$}\ %
\textsf{\textbf{%
  Harish-Chandra\\modules%
}}%
\noindent
\ {\tiny$\blacksquare$}\ %
\textsf{\textbf{%
  vector-valued modular forms%
}}
\\[.2\baselineskip]
\noindent
\textsf{\textbf{%
  MSC Primary:
  11F11
}}%
\ {\tiny$\blacksquare$}\ %
\textsf{\textbf{%
  MSC Secondary:
  11F12, 11F70
}}
\end{abstract}




\Needspace*{4em}
\addcontentsline{toc}{section}{Introduction}
\markright{Introduction}
\lettrine[lines=2,nindent=.2em]{\tbf V}{\,ector-valued} modular forms already appear in classical work by Kuga and Shi\-mu\-ra~\cite{kuga-shimura-1960}, who describe a relation to scalar-valued modular forms, which was recast in the light of quasi-modular forms by various authors~\cite{zemel-2015,choie-lee-2016}. They play a crucial role in the recent classification of Harish-Chandra modules associated with harmonic weak Maa\ss{} forms~\cite{bringmann-kudla-2018}. Without them the classification provided by Bring\-mann--Kudla would require several special cases related to the exceptional role of the Eisenstein series of weight~$2$ and level~$1$. With the help of vector-valued modular forms, its special behavior can be extended to other weights, yielding a full case in the classification that is on equal footing with the other ones. Vector-valued modular forms and vector-valued Maa\ss{} forms are thus firmly set up to endure in the theory of harmonic weak Maa\ss{} forms. The goal of the present paper is to highlight a connection between them and tensor product functors for Harish-Chandra modules that were studied by Bernstein and Gelfand~\cite{bernstein-gelfand-1980} in the 80ies, who build upon earlier work of Kostant~\cite{kostant-1975}. As applications, we relate the construction of Bringmann--Kudla to these functors and examine a holomorphic and modular analogue~$\Evec_2$ of the weight-$2$ Eisenstein series.

Throughout, we work with the algebraic group~$G = \SL{2}$ and the associated complexified Lie algebra~$\frakg = \fraksl{2}$. Specialized to this case, Bernstein--Gelfand investigated the functors~$M \mto V \otimes M$ assigning to a~$\frakg$-module~$M$ its tensor product with a finite dimensional~$\frakg$-module~$V$. They showed that the functors~$V \otimes \,\cdot\,$ provide equivalences between categories of Harish-Chandra modules of specific Harish-Chandra parameters. In our situation, possible indecomposable~$V$ are the symmetric power representations~$\sym^\sfd$ of the standard representation for nonnegative integers~$\sfd$.

For us, the key observation is that the equivalences of categories provided by tensor product functors requires a projection to specific blocks of the category. In other words, the tensor products themselves, without these projections, may have constituents in further blocks. We write~$\varpi(E_2)_\infty$ for the Harish-Chandra module associated with~$E_2$. Then Bringmann--Kudla's construction is an instance where the nontrivial extension class that is manifested in~$\varpi(E_2)_\infty$ is sent to another nontrivial one realized by a submodule of~$\sym^\sfd \otimes \varpi(E_2)_\infty$. This one is related to the equivalence of categories found by Bernstein--Gelfand. However, if~$\sfd \ge 1$ then $\sym^\sfd \otimes \varpi(E_2)_\infty$ also splits off an irreducible direct summand, which is a limit of discrete series if~$m = 1$. This is the representation theoretic essence of our construction of~$\Evec_2$, and does not participate in the equivalence studied by Bernstein--Gelfand.

The space~$\cA(G(\QQ) \backslash G(\bbA))$ of automorphic forms for~$G$ is a~$(\frakg, K_\infty)$-module and a smooth~$G(\bbAf)$-representation, where~$K_\infty = \SO{2}(\RR) \subset G(\RR)$ and~$\bbA_f$ are the finite adeles. Refining the definition given by Bringmann--Kudla, we provide spaces of vector-valued automorphic forms~$\cA(G(\QQ) \backslash G(\bbA), \sym^\sfd)$. We say that a smooth $(\frakg,K_\infty) \times G(\bbAf)$-repre\-sen\-ta\-tion is vector-valued automorphic if it occurs as a subquotient of $\cA(G(\QQ) \backslash G(\bbA), \sym^\sfd)$. Our first result, instrumental to make the connection with the functors by Bernstein--Gelfand, provides vector-valued automorphic representations that are trivial at the finite places. We thus have a natural candidate for the modules~$V$ in the formalism of Bernstein--Gelfand. The next theorem also confirms that it is compatible with the notion of vector-valued automorphic representations. The trivial representation of~$G(\QQ_v)$ for a place~$v$ of~$\QQ$ is denoted~$\bbone_v$. Restricted tensor products in this paper are always taken with respect to spherical vectors.

\begin{maintheorem}%
\label{mainthm:automorphic_tensor_products}
There are vector-valued automorphic representations
\begin{gather*}
  \varpi(\frake_{\sfd,0})
\;\subset\;
  \cA\big( G(\QQ) \backslash G(\bbA), \sym^\sfd \big)
\end{gather*}
that admit the following restricted tensor product decomposition:
\begin{gather*}
  \varpi(\frake_{\sfd,0})
\;\cong\;
  \sym^\sfd \,\otimes\,
  \sideset{}{'}{\bigotimes}_{v \ne \infty}\,
  \bbone_v
\tx{.}
\end{gather*}

Given an automorphic representation~$\varpi$, that is a subquotient of~$\cA(G(\QQ) \backslash G(\bbA))$, the tensor product~$\varpi(\frake_{\sfd,0}) \otimes \varpi$ is vector-valued automorphic.
\end{maintheorem}

The vector-valued automorphic representations~$\varpi(\frake_{\sfd,0})$ are associated with the almost holomorphic vector-valued modular forms~$\frake_{j,\sfd-j}$ that appear in the work of Bring\-mann--Kudla~\cite{bringmann-kudla-2018}, and previously in work of Zemel~\cite{zemel-2015}. It is well-known how to associate an automorphic form~$\td{f}$ to a modular form~$f$, and thus the automorphic representation~$\varpi(f)$ generated by~$\td{f}$. This procedure extends to vector-valued modular forms and we find that the representations~$\varpi(\frake_{j,\sfd-j})$ for~$0 \le j \le \sfd$ are identical. 

For cuspidal newforms~$f$, it is known that~$\varpi(f)$ decomposes as a restricted tensor product~$\bigotimes' \varpi(f)_v$ over the place~$v$ of~$\QQ$ where~$\varpi(f)_v$ is irreducible. When giving up the assumption that~$f$ is cuspidal, this no longer holds in general. For example, Roy--Schmidt--Yi~\cite{roy-schmidt-yi-2021-preprint} determined the structure of~$\varpi(E_2)$ for the modular, nonholomorphic Eisenstein series~$E_2$ of weight~$2$ and level~$1$. They found that it fits into an exact sequence
\begin{gather*}
  \bbone
\lhra
  \varpi(E_2)
\lthra
  \cD(1) \,\otimes\,
  \sideset{}{'}{\bigotimes}_{v \ne \infty}\,
  \Pi_{\frac{1}{2},v}
\tx{,}
\end{gather*}
where~$\cD(k-1)$ for positive integers~$k$ is the (limit of) holomorphic discrete series of Harish-Chandra parameter~$k-1$ and the representations~$\Pi_{1 \slash 2,v}$ are degenerate principal series for~$G(\QQ_v)$. Our second theorem exhibits the structure of the tensor product~$\varpi(\frake_{1,0}) \otimes \varpi(E_2)$.
\begin{maintheorem}%
\label{mainthm:automorphic_tensor_products_E2}
We have inclusions of internal direct sums
\begin{gather*}
  \varpi(\frake_{1,0}) \oplus \varpi(\Evec_2)
\;\subset\;
  \varpi(\frake_{1,0}) \cdot \varpi(E_2)
\;\subset\;
  \cA\big( G(\QQ) \backslash G(\bbA), \sym^1 \big)
\end{gather*}
and a short exact sequence
\begin{gather*}
  \varpi(\frake_{1,0}) \oplus \varpi(\Evec_2)
\lhra
  \varpi(\frake_{1,0}) \otimes \varpi(E_2)
\lthra
  \varpi
\end{gather*}
for a vector-valued automorphic representation
\begin{gather*}
  \varpi
\;\cong\;
  \cD(2)
  \,\otimes\,
  \sideset{}{'}{\bigotimes}_{v \ne \infty}\,
  \Pi_{\frac{1}{2},v}
\tx{.}
\end{gather*}
In particular,
\begin{gather*}
  \varpi(\Evec_2)
\;\cong\;
  \cD(0)
  \,\otimes\,
  \sideset{}{'}{\bigotimes}_{v \ne \infty}\,
  \Pi_{\frac{1}{2},v}
\tx{.}
\end{gather*}
\end{maintheorem}

The irreducible direct summand that splits off~$\varpi(\frake_{1,0}) \otimes \varpi(E_2)$ in Theorem~\ref{mainthm:automorphic_tensor_products_E2} is a natural source for our construction of~$\Evec_2$. The remaining direct summand corresponds to a vector-valued harmonic Maa\ss{} form in weight~$3$ that Bringmann--Kudla encounter in their case III(b).

In Section~\ref{sec:modular_forms}, we revisit some of the classical theory of modular forms and recall the functions~$\frake_{j,\sfd-j}$ from the work of Bringmann--Kudla including their images under the Maa\ss{} operators. In Section~\ref{sec:vector_valued_eisenstein_series}, we define the vector-valued analogue~$\Evec_2$ of the weight-$2$ Eisenstein series. We give two expressions for it: One as a linear combination of~$\frake_{j,\sfd-j}$ with coefficients in modular forms, and another one in terms of vector-valued Eisenstein series and their residues. In Section~\ref{sec:automorphic_forms}, we revisit some of the connection between modular forms and automorphic forms. We also recall the definition of principal series, restate one of the results by Roy--Schmidt--Yi, and establish a related statement on subrepresentations of the space of automorphic forms. In the final Section~\ref{sec:vector_valued_automorphic_forms}, we introduce vector-valued automorphic forms and corresponding automorphic representations. Theorem~\ref{mainthm:automorphic_tensor_products} is a combination of Theorem~\ref{thm:finite_vector_valued_automorphic_representation} and Theorem~\ref{thm:automorphic_tensor_products}, which we state and prove in this section. We conclude the paper with Theorem~\ref{thm:automorphic_tensor_products_E2}, which encompasses the statement of Theorem~\ref{mainthm:automorphic_tensor_products_E2}, and its Corollaries~\ref{cor:E2vec_automophic_representation} and~\ref{cor:BK_automophic_representation}, which present the consequences for~$\Evec_2$ and the example of Bringmann--Kudla.

\subsection*{Acknowledgement}

The author is grateful to Claudia Alfes-Neumann for fruitful discussions and comments. He thanks the Institut Mittag-Leffler, where parts of this work was conducted during the program on Moduli and Algebraic Cycles.


\section{Modular forms}
\label{sec:modular_forms}

In this section, we revisit some basic notions of the theory of scalar- and vector-valued modular forms. We also recall some of the special vector-valued modular forms that occurred in the work of Bringmann--Kudla~\cite{bringmann-kudla-2018} and Zemel~\cite{zemel-2015}.

\subsection{Preliminaries}

Given a complex number~$z$, we set~$e(z) = \exp(2 \pi i\, z)$. Throughout the text, we set~$S = \begin{psmatrix} 0 & -1 \\ 1 & 0 \end{psmatrix}$ and~$T = \begin{psmatrix} 1 & 1 \\ 0 & 1 \end{psmatrix}$, which are generators for~$\SL{2}(\ZZ)$. We let~$\Ga_\infty = \{ \pm T^n \,:\, n \in \ZZ\}$. The Poincar\'e upper half plane will be written as~$\HS$. Its elements will be commonly denoted~$\tau = x + i y$, $x, y \in \RR$. It carries an action of~$\SL{2}(\RR)$ by M\"obius transformations:
\begin{gather*}
  g \tau
=
  \mfrac{a \tau + b}{c \tau + d}
\tx{,}\quad
  \tau \in \HS
\tx{,}\,
  g = \begin{psmatrix} a & b \\ c & d \end{psmatrix} \in \SL{2}(\RR)
\tx{.}
\end{gather*}
We have a surjection~$\SL{2}(\RR) \thra \HS$ of~$\SL{2}(\RR)$\nbd sets, which maps~$g$ to~$g i$ and whose kernel is~$\SO{2}(\RR)$.

For integers~$k$, we have slash actions on functions from~$\HS$ to~$\CC$, defined by
\begin{gather*}
  \big( f \big|_k\, g \big)(\tau)
=
  (c \tau + d)^{-k} f(g \tau)
\tx{,}\quad
  g = \begin{psmatrix} a & b \\ c & d \end{psmatrix} \in \SL{2}(\RR)
\tx{.}
\end{gather*}
The compatibility between these actions and the action of~$\SL{2}(\RR)$ by right shift on functions from~$\SL{2}(\RR)$ to~$\CC$, which is standard in automorphic representation theory, will be explained in Section~\ref{sec:automorphic_forms}.

We obtain the usual notion of holomorphic modular forms of weight~$k \in \ZZ$. Using the Maa\ss{} lowering and raising operators and the Laplace operator
\begin{gather}
\label{eq:def:maass_operators}
  \rmL_k := - 2 i y^2 \partial_{\ov\tau}
\tx{,}\quad
  \rmR_k := 2 i \partial_\tau + k y^{-1}
\tx{,}\quad\tx{and}\quad
  \Delta_k := - \rmR_{k-2} \rmL_k
\tx{,}
\end{gather}
we define an almost holomorphic modular form of weight~$k$ as a smooth function~$f :\, \HS \ra \CC$ subject to the condition that
\begin{enumerateroman*}
\item
$f$ is almost holomorphic, i.e., $\rmL_k^{d+1}\, f = 0$ for some nonnegative integer~$d$.

\item
$f |_k\,\ga = f$ for every~$\ga \in \SL{2}(\ZZ)$,

\item
$|f(x + i y)| \ll y^a$ for some~$a > 0$ uniformly in~$x$ as~$y \ra \infty$.
\end{enumerateroman*}

\subsection{Vector-valued modular forms}

We define an arithmetic type as a finite dimensional, complex representation of~$\SL{2}(\ZZ)$. We write~$V(\rho)$ for the representation space of such an arithmetic type~$\rho$. We have a vector-valued slash action on functions $f :\, \HS \ra V(\rho)$ defined by
\begin{gather}
  f \big|_{k,\rho}\,\ga
:=
  \rho(\ga)^{-1}\, \big( f \big|_k\,\ga \big)
\tx{.}
\end{gather}
We say that a function~$f :\, \HS \ra V(\rho)$ is a holomorphic modular form of weight~$k$ and (arithmetic) type~$\rho$ if
\begin{enumerateroman*}
\item
$f$ is holomorphic,

\item
$f |_{k,\rho}\,\ga = f$ for all~$\ga \in \SL{2}(\ZZ)$,

\item
$\|f(x + i y)\| \ll y^a$ for some~$a > 0$ locally uniformly in~$x$ as~$y \ra \infty$ for any norm~$\| \,\cdot\, \|$ on~$V(\rho)$.
\end{enumerateroman*}
The space of such forms is written as~$\rmM_k(\rho)$.

There are analogous notions of real-analytic modular forms. In particular, we have a space~$\rmM^\ahol_k(\rho)$ of almost holomorphic modular forms, which vanish under a power of the lowering operator defined in~\eqref{eq:def:maass_operators}, and the notion of harmonic Maa\ss{} forms, which vanish under the Laplace operator~$\Delta_k$.

\paragraph{Symmetric powers}

Given a nonnegative integer~$\sfd$, we let~$\sym^\sfd$ be the symmetric power of the standard representation of~$\SL{2}$ over~$\CC$, which is an arithmetic type. When performing calculations, we will employ the realization on the space~$\CC[X]_\sfd$ of polynomials in a formal variable~$X$ of degree at most~$\sfd$ with the usual~$|_{-\sfd}$ action. In order to differentiate between the action on~$\tau$ and the action on~$X$, we will use two indices for the slash action and suppress the negative sign for~$\sfd$, writing~$|_{k,\sfd}$. Specifically, given a function~$f :\, \HS \ra \CC[X]_\sfd$, we set
\begin{gather*}
  \big( f \big|_{k,\sfd}\, g \big)(\tau)(X)
\;:=\;
  (c X + d)^\sfd
  (c \tau + d)^{-k}
  f(g \tau)(g X)
\in
  \CC[X]_\sfd
\tx{,}\quad
  g = \begin{psmatrix} a & b \\ c & d \end{psmatrix} \in \SL{2}(\RR)
\tx{.}
\end{gather*}
We record for clarity and for later use that
\begin{gather*}
  X^j \big|_{0,\sfd} S
=
  (-1)^j X^{\sfd-j}
\tx{,}
\quad
  1 |_{0,1} S = X
\tx{,}\,
\quad
  X |_{0,1} S = -1
\tx{.}
\end{gather*}
We write~$\rmM_{k,\sfd}$ and~$\rmM^\ahol_{k,\sfd}$ for~$\rmM_k(\sym^\sfd)$ and~$\rmM^\ahol_k(\sym^\sfd)$, that is, the associated spaces of holomorphic and almost holomorphic modular forms.

We have a vector-valued holomorphic modular form~$( X - \tau )^\sfd$ of weight~$-\sfd$ and arithmetic type~$\sym^\sfd$; See for example~\cite{mertens-raum-2021}. More generally, Zemel~\cite{zemel-2015} and later Bringmann-Kudla~\cite{bringmann-kudla-2018} provided a family of almost holomorphic modular forms, which goes back to at least as early as Verdier~\cite{verdier-1961}, that will be convenient in the context of raising and lowering operators. For an integer~$0 \le j \le \sfd$, we adapt the notation of Bringmann--Kudla and set
\begin{gather}
\label{eq:def:symd_LR_basis}
  \frake_{j,\sfd-j}(\tau)
:=
  \mfrac{(-1)^{\sfd-j}}{j!}
  y^{j-\sfd}
  (X - \tau)^j
  (X - \ov\tau)^{\sfd-j}
\;\in\;
  \rmM^\ahol_{\sfd-2j,\sfd}
\tx{.}
\end{gather}
By a direct calculation, we can verify the stronger modular invariance condition with respect to the real Lie group:
\begin{gather}
\label{eq:symd_LR_basis_SL2Rmodular}
  \forall g \in \SL{2}(\RR) \quantsep
  \frake_{j,\sfd-j} \big|_{\sfd-2j,\sfd} g
=
  \frake_{j,\sfd-j}
\tx{.}
\end{gather}

As an almost special case of Proposition~3.1 of~\cite{zemel-2015} in the spirit of Kuga--Shimura~\cite{kuga-shimura-1960}, we record that
\begin{gather}
\label{eq:vector_valued_almost_holomorphic_decomposition}
  \rmM^\ahol_{k,d}
\;=\;
  \bigoplus_{j = 0}^\sfd
  \frake_{j,\sfd-j} \cdot \rmM^\ahol_{k+\sfd-2j}
\tx{.}
\end{gather}
There is no simple, analogous statement for holomorphic modular forms that holds for all weights~$k$; see Theorem~3.4 of~\cite{zemel-2015}.

\subsection{Maa\ss{} operators on vector-valued modular forms}
\label{ssec:vector_valued_eisenstein_series_maass_operators}

We revisit the covariance properties of the classical Maa\ss{} operators in~\eqref{eq:def:maass_operators} in the vector-valued setting, including a sketch of a proof that circumvents the explicit verification of covariance. We claim that for all smooth functions~$f :\, \HS \ra \CC[X]_\sfd$ and all~$g \in \SL{2}(\RR)$, we have
\begin{gather}
\label{eq:symd_maass_operators_covariant}
  \rmL_k\big( f \big|_{k,\sfd}\,g \big)
\;=\;
  \big( \rmL_k\, f \big) \big|_{k-2,\sfd}\,g
\quad\tx{and}\quad
  \rmR_k\big( f \big|_{k,\sfd}\,g \big)
\;=\;
  \big( \rmL_R\, f \big) \big|_{k+2,\sfd}\,g
\tx{.}
\end{gather}

Since the lowering and raising operators are graded derivations, we have, for a smooth function~$f :\, \HS \ra \CC$ and nonnegative integers~$i$ and~$j$, the relations
\begin{align*}
  \rmL_{k+j-i} \big( \frake_{i,j} f \big)
&{}=
  \big( \rmL_{j-i}\, \frake_{i,j} \big) f
  +
  \frake_{i,j} \big( \rmL_k\, f\big) 
\tx{,}\\
  \rmR_{k+j-i} \big( \frake_{i,j} f \big)
&{}=
  \big( \rmR_{j-i}\, \frake_{i,j} \big) f
  +
  \frake_{i,j} \big( \rmR_k\, f\big) 
\tx{.}
\end{align*}
Further, Bringmann-Kudla~\cite{bringmann-kudla-2018} calculate the action of lowering and raising operators on~$\frake_{i,j}$:
\begin{gather}
\label{eq:symd_LR_basis_maass_operators}
  \rmL_{j-i} \frake_{i,j}
=
  (i+1) j \frake_{i+1,j-1}
\tx{,}\quad
  \rmR_{j-i} \frake_{i,j}
=
  \frake_{i-1,j+1}
\tx{.}
\end{gather}

Now, the modular invariance of~$\frake_{i,j}$ under~$\SL{2}(\RR)$ implies the covariance in~\eqref{eq:symd_maass_operators_covariant}, since we have for all~$g \in \SL{2}(\RR)$
\begin{alignat*}{6}
  \rmL_{j-i} \big( \frake_{i,j} \big|_{j-i,i+j} g \big)
&{}=
  \rmL_{j-i} &&\frake_{i,j}
&&{}=
  (i+1) j\, &&\frake_{i+1,j-1}
&&{}=
  (i+1) j\, &&\frake_{i+1,j-1} \big|_{j-i-2,i+j} g
\\
  \rmR_{j-i} \big( \frake_{i,j} \big|_{j-i,i+j} g \big)
&{}=
  \rmR_{j-i} &&\frake_{i,j}
&&{}=
  &&\frake_{i-1,j+1}
&&{}=
  &&\frake_{i-1,j+1} \big|_{j-i+2,i+j} g
\tx{.}
\end{alignat*}

\subsection{Real-analytic Eisenstein series}%
\label{ssec:real_analytic_eisenstein_series}

For~$k \in \ZZ$ and~$s \in \CC$, $2 \Re(s) + k > 2$, we define
\begin{gather}
  E_k(\tau, s)
\;:=\;
  \sum_{\ga \in \Ga_\infty \backslash \SL{2}(\ZZ)}
  y^s \big|_k\,\ga
\tx{.}
\end{gather}
If~$k > 2$, we set~$E_k(\tau) = E_k(\tau,0)$.

Employing Maa\ss's formula on pages~210 of~\cite{maass-1964}, which provides the Fourier expansion of the modified Eisenstein series~$y^{s+k-1}\, E_k(\tau, s)$, and Maa\ss's modified Whittaker function on page~181 of~\cite{maass-1964}, we find that 
\begin{multline}
\label{eq:classical_eisenstein_fourier_expansion}
  E_k(\tau, s)
\;=\;
  y^s
  \,+\,
  (-1)^{\frac{k}{2}}
  2^{2-k-2s}
  \pi
  \frac{\Ga(k-1+2s)}{\Ga(s) \Ga(k+s)}\,
  \frac{\zeta(k-1+2s)}{\zeta(k + 2s)}
  y^{1-k-s}
\\
  +\,
  \frac{ (-1)^{\frac{k}{2}} 2^{s+\frac{k}{2}} \pi^{k+2s}}{\zeta(k + 2s)}
  \sum_{\substack{n \in \ZZ\\n \ne 0}}
  \frac{\sigma_{k-1+2s}(|n|)}{\Ga\big(s + (1 + \sgn(n)) \frac{k}{2}\big)}\,
  y^{\frac{k}{2}}
  W_{\sgn(n)\frac{k}{2}, \frac{k-1}{2} + s}(4 \pi |n| y)\,
  e(n x)
\tx{.}
\end{multline}

Analytic continuation via the Fourier expansion allows us to evaluate at~$s = 0$ even if~$k = 2$. We set~$E_2(\tau) = E_2(\tau,0)$ and find (see also p.~19 of~\cite{bruinier-van-der-geer-harder-zagier-2008})
\begin{gather*}
  E_2(\tau)
\;=\;
  \Ehol_2(\tau)
  -
  \mfrac{3}{\pi} y^{-1}
\tx{,}\quad
  \Ehol_2(\tau)
\;=\;
  1
  \,-\,
  24 \sum_{n = 1}^\infty \sigma_1(n) e(n \tau)
\tx{.}
\end{gather*}
Observe that~$\Ehol_2$ is not a modular form. Instead, it is both a quasi-modular form and a mock modular form; One of the few instances where these two notions overlap. Correspondingly, $E_2$ is both an almost holomorphic modular form and a harmonic Maa\ss{} form.

\section{Vector-valued Eisenstein series}
\label{sec:vector_valued_eisenstein_series}

In this section, we define the vector-valued analogue~$\Evec_2$ of the quasi-modular holomorphic Eisenstein series in weight~$2$ and provide two constructions that reveal its modularity. In~\eqref{eq:E2vec_LR_basis}, we express it as a linear combination of Bringmann--Kudla's~$\frake_{j,1-j}$ with coefficients in almost holomorphic modular forms. In Sections~\ref{ssec:vector_valued_real_analytic_eisenstein_series}, we define real-analytic vector-valued Eisenstein series. This allows us to express~$\Evec_2$ in terms of a residue of Eisenstein series and an Eisenstein series.

\subsection{An analogue of the weight-$2$ Eisenstein series}
\label{ssec:vector_valued_E2}

We define a vector-valued modular form that combines features of~$E_2$ and~$\Ehol_2$. On the one hand it is modular, on the other hand it is holomorphic. It has weight~$1$, but nevertheless we write~$\Evec_2$ in order to emphasize the connection to~$E_2$ and $\Ehol_2$.
\begin{gather}
\label{eq:def:vector_valued_E2hol}
  \Evec_2(\tau)
\;:=\;
  (X - \tau) \Ehol_2(\tau) - \mfrac{6}{\pi i}
\in
  \CC[X]_1
\tx{.}
\end{gather}
We can recover~$\Ehol_2$ as the~$X$\nbd component of the vector-valued holomorphic modular form~$\Evec_2$.

It is also helpful to express~$\Evec_2$ in terms of the basis~$\frake_{j,1-j}(\tau)$. We have
\begin{gather}
\label{eq:E2vec_LR_basis}
  \Evec_2(\tau)
\;=\;
  \frake_{1,0}(\tau)
  E_2(\tau)
  -
  \frake_{0,1}(\tau)
  \mfrac{3}{\pi}
\;=\;
  (X - \tau)
  \big( \Ehol_2(\tau) - \mfrac{3}{\pi} y^{-1}\big)
  -
  \big( - y^{-1} (X - \ov\tau) \big)
  \mfrac{3}{\pi}
\tx{.}
\end{gather}
This manifests the connection to the modular Eisenstein series~$E_2$.

\begin{proposition}
We have
\begin{gather*}
  \Evec_2 \in \rmM_{1,1}
\tx{.}
\end{gather*}
\end{proposition}
\begin{proof}
We have~$\Evec_2 \in \rmM^\ahol_{1,1}$ by~\eqref{eq:E2vec_LR_basis}. Further, $\Evec_2$ is holomorphic by its defining expression~\eqref{eq:def:vector_valued_E2hol}.
\end{proof}

\begin{remark}
It is instructive in the context of Section~\ref{ssec:automorphic_tensor_products} to verify that~$\Evec_2$ is holomorphic from the expression in~\eqref{eq:E2vec_LR_basis}:
\begin{gather*}
  \rmL_1 \Evec_2
=
  \rmL_1 \big( \frake_{1,0} E_2 - \frake_{0,1} \mfrac{3}{\pi} \big)
=
  \frake_{1,0} \mfrac{3}{\pi} - \frake_{1,0} \mfrac{3}{\pi}
=
  0
\tx{.}
\end{gather*}
\end{remark}

\subsection{Real-analytic Eisenstein series}
\label{ssec:vector_valued_real_analytic_eisenstein_series}

We extend the real-analytic Eisenstein series in Section~\ref{ssec:real_analytic_eisenstein_series} to the vector-valued setting and relate it to the former. For a weight~$k \in \ZZ$, $\sfd \in \ZZ_{\ge 0}$ with $k \equiv \sfd \,\pmod{2}$, $j \in \ZZ$ with~$0 \le j \le \sfd$, and~$s \in \CC$ with~$2 \Re(s) + k > 2 + \sfd$, we define
\begin{gather}
\label{eq:def:vector_valued_real_analytic_eisenstein_series}
  E_{k,\sfd}(\tau,j,s)
\;:=\;
  \sum_{\ga \in \Ga_\infty \backslash \SL{2}(\ZZ)}
  (X - \tau)^j y^s \big|_{k,\sfd}\, \ga
\tx{.}
\end{gather}
Obverse that~$X - \tau$ is invariant under the action of~$\begin{psmatrix} 1 & b \\ 0 & 1 \end{psmatrix}$, $b \in \RR$, for any~$k$ and~$\sfd$. The parity condition on~$k$ and~$\sfd$ ensures invariance under~$\pm \begin{psmatrix} 1 & 0 \\ 0 & 1 \end{psmatrix}$. Hence the action of~$\ga$ in the summand only depends on~$\Ga_\infty \ga$. The condition on~$\Re(s)$ ensures absolute convergence by usual estimates.

To provide an analytic continuation of~$E_{k,d}(\,\cdot\,,j,s)$ we contrast it with a type of Eisenstein series that directly emerges from~\eqref{eq:vector_valued_almost_holomorphic_decomposition}. We set
\begin{gather}
\label{eq:def:vector_valued_eisenstein_series_product}
  E^\ahol_{k,\sfd}(\tau, j, s)
\;:=\;
  \sum_{\ga \in \Ga_\infty \backslash \SL{2}(\ZZ)}
  \frake_{j,\sfd-j} y^s \big|_{k,\sfd}\, \ga
\tx{.}
\end{gather}
As opposed to the Eisenstein series in~\eqref{eq:def:vector_valued_real_analytic_eisenstein_series}, the one in~\eqref{eq:def:vector_valued_eisenstein_series_product} is in general not holomorphic even for~$s = 0$, as~$\frake_{j,\sfd-j}$ is merely almost holomorphic. We can, however, readily identify them as products of~$\frake_{j,\sfd-j}$ with classical Eisenstein series.

\begin{proposition}
\label{prop:vector_valued_eisenstein_series_product}
Given integers~$k$, $\sfd \ge 0$, and~$0 \le j \le \sfd$, and a complex number~$s$ with~$2 \Re(s) + k > 2 + \sfd$, we have
\begin{gather}
\label{eq:prop:vector_valued_eisenstein_series_product}
  E^\ahol_{k,\sfd}(\tau, j, s)
\;=\;
  \frake_{j,\sfd-j} \cdot E_{k-\sfd+2j}(\tau,s)
\tx{.}
\end{gather}
In particular, the Eisenstein series in~\eqref{eq:def:vector_valued_eisenstein_series_product} admits an analytic continuation to~$s \in \CC$ with at most simple poles.
\end{proposition}
\begin{proof}
We use the modular $\SL{2}(\RR)$-invariance of~$\frake_{j,\sfd-j}$ in~\eqref{eq:symd_LR_basis_SL2Rmodular}. Then after inserting the definition of the Eisenstein series, we find that
\begin{align*}
&
  E^\ahol_{k,\sfd}(\tau, j, s)
\;=\;
  \sum_{\ga \in \Ga_\infty \backslash \SL{2}(\ZZ)}
  \frake_{j,\sfd-j} y^s \big|_{k,\sfd}\, \ga
\\
\;=\;{}&
  \sum_{\ga \in \Ga_\infty \backslash \SL{2}(\ZZ)}
  \big( \frake_{j,\sfd-j} \big|_{\sfd-2j,\sfd}\, \ga \big)
  \big( y^s \big|_{k+\sfd-2j,\sfd}\, \ga \big)
\;=\;
  \sum_{\ga \in \Ga_\infty \backslash \SL{2}(\ZZ)}
  \frake_{j,\sfd-j}\,
  \big( y^s \big|_{k-\sfd+2j,\sfd}\, \ga \big)
\\
\;=\;{}&
  \frake_{j,\sfd-j}
  \sum_{\ga \in \Ga_\infty \backslash \SL{2}(\ZZ)}
   y^s \big|_{k-\sfd+2j}\, \ga
\;=\;
  \frake_{j,\sfd-j} \cdot E_{k-\sfd+2j}(\tau,s)
\tx{.}
\end{align*}
\end{proof}

The Eisenstein series~$E^\ahol_{k,\sfd}(\tau,j,s)$ and~$E_{k,\sfd}(\tau,j,s)$ can be related to each other via an explicit formula, which we make precise in the next statement.
\begin{proposition}
\label{prop:vector_valued_real_analytic_eisenstein_series_as_product}
Given integers~$k$, $\sfd \ge 0$, and~$0 \le j \le \sfd$, and a complex number~$s$ with~$2 \Re(s) + k > 2 + \sfd$, we have
\begin{gather}
\label{eq:prop:vector_valued_real_analytic_eisenstein_series_as_product}
\begin{aligned}
  E_{k,\sfd}(\tau, j, s)
&{}=
  \big( \mfrac{i}{2} \big)^{\sfd-j}
  \sum_{r = j}^\sfd
  \mbinom{\sfd-j}{r-j}
  r!\,
  \frake_{r,\sfd-r}
  E_{k-\sfd+2r-2j}(\tau, s-r+j)
\\
&{}=
  \big( \mfrac{i}{2} \big)^{\sfd-j}
  \sum_{r = j}^\sfd
  \mbinom{\sfd-j}{r-j}
  r!\,
  E^\ahol_{k,\sfd}(\tau,r,s-r+j)
\tx{.}
\end{aligned}
\end{gather}
In particular, the Eisenstein series in~\eqref{eq:def:vector_valued_real_analytic_eisenstein_series} admits an analytic continuation to~$s \in \CC$ with at most simple poles.
\end{proposition}
\begin{proof}
This follows when expanding~$(X-\tau)^j 1^{\sfd-j}$ according to the following equations:
\begin{gather*}
  1
=
  \mfrac{-1}{2iy} \big( (X - \tau) - (X - \ov{\tau}) \big)
\tx{,}
\end{gather*}
and hence
\begin{align*}
&
  (X - \tau)^j 1^{\sfd-j}
\;=\;
  \big( \mfrac{i}{2y} \big)^{\sfd-j}
  (X - \tau)^j
  \big( (X - \tau) - (X - \ov{\tau}) \big)^{\sfd-j}
\\
=\;{}&
  \big( \mfrac{i}{2y} \big)^{\sfd-j}
  \sum_{r = 0}^{\sfd-j}
  \mbinom{\sfd-j}{r}
  (-1)^{\sfd-j-r}
  (X - \tau)^{j+r}
  (X - \ov{\tau})^{\sfd-j-r}
\\
=\;{}&
  \big( \mfrac{i}{2} \big)^{\sfd-j}
  \sum_{r = 0}^{\sfd-j}
  \mbinom{\sfd-j}{r}
  (j+r)!
  y^{-r}
  \frake_{j+r,\sfd-j-r}
\tx{.}
\end{align*}
\end{proof}

We are now in position, to express~$\Evec_2$ via vector-valued Eisenstein series.
\begin{proposition}
We have
\begin{gather}
\label{eq:Evec2_eisenstein_expression}
  \Evec_2(\tau)
\;=\;
  E_{1,1}(\tau,1,0)
  \,+\,
  2 i\, \Res_{s=1}\,E_{1,1}(\tau,0,s)
\tx{.}
\end{gather}
\end{proposition}
\begin{proof}
Note that Proposition~\ref{prop:vector_valued_real_analytic_eisenstein_series_as_product} immediately implies that
\begin{gather*}
  E_{1,1}(\tau,1,0)	
\;=\;
  \frake_{1,0}(\tau)
  E_2(\tau)
\tx{.}
\end{gather*}

We also obtain the expression
\begin{gather*}
  E_{1,1}(\tau,0,s)
=
  \mfrac{i}{2}
  \big(
  - \frake_{1,0} E_2(\tau,s-1)
  + \frake_{0,1} E_0(\tau,s)
  \big)
\tx{,}
\end{gather*}
which reveals a simple pole at~$s = 1$. Recall from Miyake's Corollary~7.2.10 in~\cite{miyake-1989} (or employ~\eqref{eq:classical_eisenstein_fourier_expansion}) that
\begin{gather*}
  \Res_{s = 1}\, E_0(\tau,s)
=
  \mfrac{\pi}{2 \zeta(2)}
=
  \mfrac{3}{\pi}
\tx{,}
\end{gather*}
where care must be taken that he defines modified real-analytic Eisenstein series in~(7.2.1), which lack the factor~$y^s$ and are normalized by an additional factor~$\zeta(2s+k)$. This yields
\begin{gather*}
  \Res_{s = 1}\, E_{1,1}(\tau,0,s)
\;=\;
  \mfrac{i}{2}
  \frake_{0,1}\,
  \Res_{s = 1} E_0(\tau,s)
=
  \mfrac{3 i}{2 \pi}
  \frake_{0,1}
\tx{.}
\end{gather*}

Combining this with~\eqref{eq:E2vec_LR_basis}, we conclude that 
\begin{gather*}
  \Evec_2(\tau)
\;=\;
  E_2(\tau) \frake_{1,0}(\tau)
  -
  \mfrac{3}{\pi} \frake_{0,1}(\tau)
\;=\;
  E_{1,1}(\tau,1,0)
  \,+\,
  2 i\, \Res_{s=1}\,E_{1,1}(\tau,0,s)
\tx{.}
\end{gather*}
\end{proof}

\section{Automorphic forms}%
\label{sec:automorphic_forms}

In this section, we revisit some of theory of automorphic forms, principal series, and automorphic Eisenstein series. We base this section on recent work by Roy--Schmidt--Yi~\cite{roy-schmidt-yi-2021-preprint}, who examined the automorphic representation associated with~$E_2$. The content of this section is mostly preparatory for Section~\ref{sec:vector_valued_automorphic_forms}. The only statement not contained in the literature is given in Section~\ref{ssec:modular_automorphic_eisenstein_series_wt2}.

\subsection{Preliminaries}%
\label{ssec:automorphic_forms:preliminaries}

We write~$\QQ_p$ and~$\ZZ_p$ for the rings of $p$-adic rationals and integers.  We write~$\bbA$ for the adeles of~$\QQ$, and $\bbAf \subset \bbA$ for the finite adeles. To shorten notation and make it slightly more compatible with common use in automorphic representation theory, we will write~$G$ for the algebraic group~$\SL{2}$, $P$ for its parabolic subgroup of upper triangular matrices, $K_\infty = \SO{2}(\RR)$, $K_p = G(\QQ_p)$ for primes~$p$, $K_\rmf = \prod_p K_p$, and~$K = K_\rmf \times K_\infty$. We write~$\frakg$ for the complexified Lie algebra of~$G(\RR)$ and~$\frakz$ for the center of the universal enveloping algebra~$\frakU(\frakg)$ of~$\frakg$.

To emphasize the connection between modular and automorphic forms, we usually employ the following notation
\begin{gather}
\label{eq:parabolic_element_coordinates}
  p
\;=\;
  \begin{psmatrix} \td{y} & x \td{y}^{-1} \\ 0 & \td{y}^{-1} \end{psmatrix}
\in
  P(\QQ_v)
\tx{,}
\end{gather}
where at the infinite place we assume that~$\td{y} > 0$ and we write~$p(\tau)$ with~$\tau = x + i \td{y}^2$, if we need to emphasize the coordinates.

We recall the definition of automorphic forms, adopted in a slightly different form from~\cite{borel-jacquet-1979}. An automorphic form on~$G$ is a function~$\td{f} : G(\bbA) \ra \CC$ such that~$\td{f}(\ga g) = \td{f}(g)$ for all~$\ga \in G(\QQ)$, $g \in G(\bbA)$, $\td{f}$ is~$K$-finite for the action by right shifts, $\td{f}$ is $\frakz$-finite, and for every norm~$\|\,\cdot\,\|$ on~$G(\bbA)$ in the sense of~\cite{borel-jacquet-1979}, there is~$a > 0$ such that~$|\td{f}(g)| \ll \|g\|^a$ for all~$g \in G(\bbA)$. The space of automorphic forms~$\cA(G(\QQ) \backslash G(\bbA))$ is a module for~$(\frakg,K_\infty) \times G(\bbAf)$. By construction this is a smooth representation. Automorphic representations are defined as constituents of~$\cA(G(\QQ) \backslash G(\bbA))$.

We write~$\cD(k-1)$, $k \ge 1$ an integer, for the (limit of) holomorphic discrete series of Harish-Chandra parameter~$k-1$. Its lowest weight has index~$k - 1 + \frac{1+1}{2} = k$. Similarly, $\ov{\cD}(k-1)$ denotes the (limit of) anti-holomorphic discrete series. The Steinberg representation for~$\SL{2}(\QQ_v)$ will written as~$\mathrm{St}_v$. We write the trivial representation of~$G(\QQ_v)$ as~$\bbone_v$.

\paragraph{Modular forms}

We next recall the procedure by which one associates an automorphic form~$\td{f}$ to an almost holomorphic modular form~$f$ of weight~$k$, or more generally a real-analytic one that vanishes under a polynomial in the Laplace operator. We define
\begin{gather}
  \td{f}(\ga g_\infty k_\rmf)
\;:=\;
  \big(f \big|_k g_\infty \big)(i)
\tx{,}\quad
  \ga \in G(\QQ),\,
  g_\infty \in G(\RR),\,
  k_\rmf \in K_\rmf
\tx{.}
\end{gather}
Since every~$g \in G(\bbA)$ can be decomposed as a product~$\ga g_\infty k_\rmf$, this provides all values of~$\td{f}$. Observe that~$\td{f}$ depends on the weight~$k$, which is only implicitly given via~$f$.

While~$\td{f}$ is a~$K_\rmf$\nbd fixed vector by construction, its~$K_\infty$\nbd type is readily calculated. For all~$k = \begin{psmatrix} d & -c \\ c & d \end{psmatrix} \in K_\infty$, we have
\begin{gather*}
  \td{f}(g k_\infty)
\;:=\;
  \big(f \big|_k g k_\infty \big)(i)
\;:=\;
  (c i + d)^{-k}
  \big(f \big|_k g \big)(k_\infty i)
\;:=\;
  \td{f}(g)\,
  (c i + d)^{-k}
\tx{.}
\end{gather*}
We write~$\varpi(f)$ for the representation in~$\cA(G(\QQ) \backslash G(\bbA))$ generated by~$\td{f}$.

\subsection{Principle series}

Langlands proved in the appendix to work by Borel--Jacquet~\cite{borel-jacquet-1979} that automorphic representations are exactly the constituents of global principal series, which we define in this section. Specifically, we recall the normalized induction from~$P$ to~$G$, for which we will need the modular character~$\delta_P$ of~$P(\bbA)$.

Given a place~$v$ of~$\QQ$, let~$\mu_{P,v}$ be the normalized right Haar measure on~$P(\QQ_v)$. Then for~$p \in P(\QQ_v)$, we define~$\delta_{P,v}(p)$ by
\begin{gather*}
  \delta_{P,v}(p)\, \mu_{P,v}
\;:=\;
  p^\ast \big( \mu_{P,v} \big)
\tx{,}
\end{gather*}
where~$p^\ast$ denotes the pullback under inverse multiplication from the left. More concretely, for~$p$ as in~\eqref{eq:parabolic_element_coordinates}, we have~$\delta_P(p) = |\td{y}|^2_v$. At the finite places~$v$, we can verify this by comparing the measure of~$P(\ZZ_v)$ and its image under~$p^{-1}$:
\begin{gather*}
  \mu\big( p^{-1} P(\ZZ_v) \big)
\;=\;
  \mu\Bigg( \Bigg\{
  \begin{pmatrix}
  \td{y} a & \td{y} b + x \td{y}^{-1} a^{-1} \\
  0 & \td{y}^{-1} a^{-1}
  \end{pmatrix}
  \,:\,
  a \in \ZZ_p^\times,\, b \in \ZZ_p
  \Bigg\} \Bigg)
\;=\;
  |\td{y}|_v\, |\td{y}|_v
\;=\;
  |y|_v
\tx{.}
\end{gather*}
At the infinite place, a similar computation with, say, $a \in [\frac{1}{2}, 2]$ and~$b \in [-1,1]$ works. The global modular function~$\delta_P$, defined in an analogous way, equals the product of its local contributions~$\delta_{P,v}$ over all places of~$\QQ$.

Note that~$G(\QQ_p)$ is unimodular, i.e., its modular character is trivial. Local normalized induction of a representation~$\sigma$ of~$P(\QQ_p)$ can then be defined by the action via right shifts on the space
\begin{gather}
  V\big( \Ind_{P(\QQ_v)}^{G(\QQ_v)}\,\sigma \big)
\;:=\;
  \big\{
  f :\, G(\QQ_v) \ra V(\sigma) \,:\,
  f \tx{\ smooth},\,
  f(p g) = \delta_{P,v}(p)^{\frac{1}{2}} \sigma(p) f(g)
  \big\}
\tx{.}
\end{gather}
The representation~$\sigma$ in this work is merely a character on the Levi factor. In order to accommodate the relation to~$\PGL{2}$, we write~$|\,\cdot\,|_p^s \times |\,\cdot\,|_p^{-s}$ for the character that maps~$p$ as in~\eqref{eq:parabolic_element_coordinates} to~$|y|^s = |\td{y}|^{2s}$. We write~$\Pi_{s,v}$ for its induction if~$v \ne \infty$ and~$\Pi_{1 \slash 2, \infty}$ for the~$K_\infty$-finite vectors of the induction at~$v = \infty$.

The important case for us will be~$s = \frac{1}{2}$. If~$v = \infty$, the principal series fits into the following exact sequence of Harish-Chandra modules:
\begin{gather}
  \ov{\cD}(1) \oplus \cD(1)
\lhra
  \Pi_{\frac{1}{2},\infty}
\lthra
  \sym^1
\tx{.}
\end{gather}
If~$v$ is finite, we have an analogous sequence of~$G(\QQ_v)$-modules
\begin{gather}
  \mathrm{St}_v
\lhra
  \Pi_{\frac{1}{2},v}
\lthra
  \bbone_v
\tx{.}
\end{gather}

\subsection{The Eisenstein series of weight $2$}%
\label{ssec:modular_automorphic_eisenstein_series_wt2}

Roy--Schmidt--Yi~\cite{roy-schmidt-yi-2021-preprint} determined the exact structure of~$\varpi(E_2)$ in the context of automorphic forms for~$\GL{2}$ in their Theorem~5.11. Since the restriction of the Steinberg representation from~$\GL{2}(\QQ_v)$ to~$\SL{2}(\QQ_v)$ does not decompose further, we can adopt their result and find that~$\varpi(E_2)$ sits in the short exact sequence
\begin{gather}
\label{eq:automorphic_form_E2_short_exact_sequence}
  \bbone
\lhra
  \varpi(E_2)
\lthra
  \cD(1)
  \,\otimes\,
  \sideset{}{'}{\bigotimes}_{v \ne \infty}\,
  \Pi_{\frac{1}{2},v}
\tx{.}
\end{gather}
Note that the discrete series~$\cD(1)$ in this sequence is the one for~$\SL{2}(\RR)$ and thus different from the one in~\cite{roy-schmidt-yi-2021-preprint} for~$\GL{2}(\RR)$.

The maximal irreducible quotient~$\ov\varpi(E_2)$ of~$\varpi(E_2)$ can be reasonably viewed as the automorphic representation associated with~$\Ehol_2$. We want to verify that~$\ov\varpi(E_2)$ is not a subrepresentation of~$\cA(G(\QQ) \backslash G(\bbA))$. While this is folklore, to our knowledge it is not stated in the required form in the available literature.

\begin{proposition}
\label{prop:E2_automorphic_embedding}
There is \emph{no} intertwining inclusion of\/~$\ov\varpi(E_2)$ into~$\cA(G(\QQ) \backslash G(\bbA))$.
\end{proposition}

\begin{proof}
Assuming the contrary, we have a map
\begin{gather}
\label{eq:E2_automorphic_embedding}
  \ov\varpi(E_2)
\lhra
  \cA(G(\QQ) \backslash G(\bbA))
\tx{.}
\end{gather}
Let~$\phi = \bigotimes_v \phi_v \in \ov\varpi(E_2)$ be the spherical vector at all finite places~$v$ and a generator of the lowest~$K_\infty$\nbd type at the infinite place. We write~$\td{f}$ for the image of~$\phi$ under~\eqref{eq:E2_automorphic_embedding} and define a function~$f : \HS \ra \CC$ by
\begin{gather*}
  \big( f \big|_2\, g_\infty \big)(i)
\;:=\;
  \td{f}(g_\infty)
\tx{,}\quad
  g_\infty \in G(\RR)
\tx{.}
\end{gather*}
This is well-defined since~$\td{f}$ has~$K_\infty$-type corresponding to modular weight~$2$. It is invariant under~$\Ga$, since~$\td{f}$ is automorphic and spherical at the finite places. Since~$\phi_\infty$ is a lowest weight vector, $f$ is holomorphic. At the finite places~$\phi_v$ has the same Hecke eigenvalue~$1 + p$ as~$E_2$. The growth condition for~$\td{f}$ implies that~$f$ has moderate growth. From the Hecke eigenvalues we conclude that~$f$ is a multiple of~$\Ehol_2$ up to an additive constant, which is not modular; A contradiction.
\end{proof}

\section{Vector-valued automorphic forms}%
\label{sec:vector_valued_automorphic_forms}

In this section we define the space~$\cA(G(\QQ) \backslash G(\bbA), \rho)$ of vector-valued automorphic forms. We introduce finite dimensional, vector-valued automorphic representations in Section~\ref{ssec:finite_vector_valued_automorphic_representation}.
Then in Section~\ref{ssec:automorphic_tensor_products} we relate them to the tensor product functors studied by Bernstein--Gelfand. This allows us to reinterpret the vector-valued automorphic representation~$\varpi(\Evec_2)$ generated by~$\Evec_2$. The tensor product functors partially split the extension class of~$\varpi(E_2)$.

\subsection{Preliminaries}

Recall the notation set up in Section~\ref{ssec:automorphic_forms:preliminaries}. We start by defining vector-valued automorphic forms following Bringmann--Kudla~\cite{bringmann-kudla-2018}. We translate their notion to the adelic setting and amend the condition that vector-val\-ued automorphic forms are~$\frakz$-finite, which appears in the definition of usual automorphic forms.

We let~$\rho$ be an arithmetic type, that is, a representation of~$\SL{2}(\ZZ) = G(\ZZ)$, that extends to~$G(\QQ)$. For~$\ga \in G(\QQ)$ embedded diagonally into~$G(\bbA)$, we set~$\rho(\ga) = \rho(\ga_\infty)$, where~$\ga = \ga_\rmf \ga_\infty$ with~$\ga_\rmf \in G(\bbAf)$ and~$\ga_\infty \in G(\RR)$. A vector-valued automorphic form on~$G$ of arithmetic type~$\rho$ is a function~$\td{f} : G(\bbA) \ra V(\rho)$ such that~$\td{f}(\ga g) = \rho(\ga) \td{f}(g)$ for all~$\ga \in G(\QQ)$, $g \in G(\bbA)$, $\td{f}$ is~$K$-finite for the action by right shifts, $\td{f}$ is $\frakz$-finite, and for every norm~$\|\,\cdot\,\|$ on~$G(\bbA)$ and every norm~$\|\,\cdot\,\|$ on~$V(\rho)$ there is~$a > 0$ such that~$\|\td{f}(g)\| \ll \|g\|^a$ for all~$g \in G(\bbA)$. The corresponding space of vector-valued automorphic forms~$\cA(G(\QQ) \backslash G(\bbA), \rho)$ is a module for~$(\frakg,K_\infty) \times G(\bbAf)$. As in the automorphic case, it is a smooth representation by construction.

\paragraph{Modular forms}

Similar to classical modular forms, one can pass from vector-valued modular forms to vector-valued automorphic forms. In the real-analytic case we need to assume that it vanishes under a polynomial in the Laplace operator. Given a vector-valued modular form~$f$ of weight~$k$ and type~$\rho$ where~$\rho$ extends to~$\SL{2}(\QQ)$ we can associate a vector-valued automorphic form as follows:
\begin{gather}
  \td{f}(\ga g_\infty k_\rmf)
\;:=\;
  \rho(\ga)\,
  \big(f \big|_k\, g_\infty \big)(i)
\tx{,}\quad
  \ga \in G(\QQ),\,
  g_\infty \in G(\RR),\,
  k_\rmf \in K_\rmf
\tx{.}
\end{gather}
The defining properties of vector-valued modular forms can all be directly verified provided that the right hand side is well-defined. For clarity, we emphasize that the right hand side does not agree with the vector-valued slash action that appears in the definition of vector-valued modular forms.

To verify that the right hand side is well-defined, we consider~$\ga \in G(\QQ) \cap G(\RR) K_\rmf$ and write~$\ga = \ga_\rmf \ga_\infty$ with~$\ga_\rmf \in K_\rmf$ and~$\ga_\infty \in G(\RR)$. Then for~$\ga' \in G(\QQ)$, $g_\infty \in G(\RR)$, and~$k_\rmf \in K_\rmf$, we have
\begin{align*}
  \td{f}\big( \ga' g_\infty k_\rmf \big)
&{}=
  \td{f}\big( \ga' \ga^{-1}\, \ga_\infty g_\infty\, \ga_\rmf k_\rmf \big)
=
  \rho\big( \ga' \ga^{-1} \big)\,
  \big( f \big|_k\, \ga_\infty g_\infty \big)(i)
\\
&{}=
  \rho\big( \ga' \ga^{-1} \big)\,
  \rho(\ga_\infty)\,
  \big(  f \big|_k\, g_\infty \big)(i)
=
  \rho\big( \ga' \big)\,
  \big(  f \big|_k\, g_\infty \big)(i)
\tx{.}
\end{align*}

We write~$\varpi(f)$ for the representation in~$\cA(G(\QQ) \slash G(\bbA), \rho)$ generated by~$\td{f}$.

\subsection{Finite dimensional representations}%
\label{ssec:finite_vector_valued_automorphic_representation}

Recall~$\frake_{j,\sfd-j}$ from Section~\ref{eq:def:symd_LR_basis}, which is a vector-valued modular form of weight~$\sfd-2j$ and type~$\sym^\sfd$. The sole purpose of this section is to determine the associated vector-valued automorphic representations.

\begin{theorem}%
\label{thm:finite_vector_valued_automorphic_representation}
Given a nonnegative integer~$\sfd$, we have~$\varpi(\frake_{j,\sfd-j}) = \varpi(\frake_{\sfd,0})$ for every integer~$0 \le j \le \sfd$. Further, we have the restricted tensor product decomposition
\begin{gather}
\label{eq:finite_vector_valued_automorphic_representation_tensor_decomposition}
  \varpi(\frake_{\sfd,0})
\;\cong\;
  \sym^\sfd \,\otimes\, \bigotimes_{v \ne \infty} \bbone_v
\tx{,}
\end{gather}
where~$\sym^\sfd$ is the~$\sfd+1$-dimensional irreducible~$(\frakg,K_\infty)$-module and~$\bbone_v$ is the trivial representation of\/~$G(\QQ_v)$.
\end{theorem}

\begin{remark}
The isomorphism can be concretely realized by the evaluation map that sends~$f \in \varpi(\frake_{j,\sfd-j})$ to~$f(1) \in \CC[X]_\sfd$. To see that it is injective, note that if~$f(1) = 0$ for~$f \in \varpi(\frake_{j,\sfd-j})$, then for all~$g_\infty \in G(\RR)$, we have~$f(g) = \sym^\sfd(g_\infty) f(1) = 0$ by~\eqref{eq:finite_vector_valued_automorphic_representation_values} in the next proof.
\end{remark}

\begin{proof}
Using the~$\SL{2}(\RR)$-covariance in~\eqref{eq:symd_LR_basis_SL2Rmodular}, we find that the vector-valued automorphic form associated to~$\frake_{j,\sfd-j}$ is
\begin{gather}
\label{eq:finite_vector_valued_automorphic_representation_values}
  \td\frake_{j,\sfd-j}\big( \ga g_\infty k_\rmf \big)
\;=\;
  \sym^\sfd\big( \ga_\infty g_\infty \big)\,
  \frake_{j,\sfd-j}(i)
\tx{,}
\end{gather}
where~$\ga = \ga_\infty \ga_\rmf \in G(\QQ)$, $\ga_\infty, g_\infty \in G(\RR)$, $\ga_\rmf \in G(\bbAf)$, and~$k_\rmf \in K_\rmf$. In particular, we discover that~$\td\frake_{j,\sfd-j}(g) = \td\frake_{j,\sfd-j}(g_\infty)$ for~$g = g_\infty g_\rmf \in G(\bbA)$ with~$g_\infty \in G(\RR)$ and~$g_\rmf \in G(\bbAf)$.

Since~$\varpi(\frake_{j,\sfd-j})$ is defined as the representation generated by~$\td{\frake}_{j,\sfd-j}$, we conclude that~$f(g_\infty g_\rmf) \;=\; f(g_\infty)$ for all~$f \in \varpi(\frake_{j,\sfd-j})$. This shows that~$\varpi(\frake_{j,\sfd-j})$ is trivial as a representation of~$G(\QQ_v)$ for all places~$v \ne \infty$ of~$\QQ$.

Since the~$\frake_{j,\sfd-j}(i)$ form a basis of~$\CC[X]_\sfd$ as~$j$ runs through integers~$0 \le j \le \sfd$, we find as we claimed that
\begin{gather*}
  \varpi(\frake_{j,\sfd-j})
=
  \varpi(\frake_{\sfd,0})
=
  \lspan\, \CC \big\{ \td\frake_{j',\sfd-j'} : 0 \le j' \le \sfd \big\}
\quad\tx{and}\quad
  \varpi(\frake_{\sfd,0})
\;\cong\;
  \sym^\sfd
\tx{.}
\end{gather*}
\end{proof}

\subsection{Automorphic tensor products}%
\label{ssec:automorphic_tensor_products}

In general, constituents of tensor products of automorphic forms are not automorphic. In this section, we show that the tensor product with~$\varpi(\frake_{\sfd,0})$ yields maps from automorphic forms to vector-valued automorphic forms. This yields an automorphic version of the tensor product functors by Bernstein--Gelfand. We examine the structure of~$\varpi(\frake_{j,\sfd-j}) \otimes \varpi(E_2)$ and relate it to one of the cases in the classification by Bringmann--Kudla~\cite{bringmann-kudla-2018} and to~$\varpi(\Evec_2)$.

\begin{theorem}
\label{thm:automorphic_tensor_products}
For all nonnegative integers~$\sfd$ and~$0 \le j \le \sfd$, we have an embedding of\/~$G(\bbAf)$-representations
\begin{gather}
\label{eq:thm:automorphic_tensor_products:embedding}
  \cA\big( G(\QQ) \backslash G(\bbA) \big)
\lhra
  \cA\big( G(\QQ) \backslash G(\bbA), \sym^\sfd \big)
\tx{,}\;
  f \lmto \td\frake_{j,\sfd-j} \cdot f
\tx{.}
\end{gather}
In particular, if~$\varpi$ is an automorphic representation that appears as a subrepresentation, i.e., $\varpi \subseteq \cA(G(\QQ) \backslash G(\bbA))$, then
\begin{gather}
\label{eq:thm:automorphic_tensor_products:product}
  \varpi(\frake_{\sfd,0}) \otimes \varpi
\;\cong\;
  \varpi(\frake_{\sfd,0}) \cdot \varpi
\;\subseteq\;
  \cA\big( G(\QQ) \backslash G(\bbA), \sym^\sfd \big)
\tx{.}
\end{gather}

More generally, if~$\varpi$ is an automorphic representation for~$G(\QQ)$, then~$\varpi(\frake_{\sfd,0}) \otimes \varpi$ is vector-valued automorphic of type~$\sym^\sfd$.
\end{theorem}
\begin{proof}
The embedding in~\eqref{eq:thm:automorphic_tensor_products:embedding} follows when verifying the definition of vector-valued automorphic forms for~$\td\frake_{j,\sfd-j} \cdot f$. For the isomorphism in~\eqref{eq:thm:automorphic_tensor_products:product}, it suffices to note that the~$\td\frake_{j,\sfd-j}(1)$ form a basis of~$\CC[X]_\sfd$.

To establish the final statement of the theorem, choose automorphic representations~$\td\varpi, \td\varpi_0 \subseteq \cA(G(\QQ) \backslash G(\bbA))$ with
\begin{gather*}
  \td\varpi_0
\lhra
  \td\varpi
\lthra
  \varpi
\tx{.}
\end{gather*}
Since~$\varpi(\frake_{\sfd,0})$ is finite dimensional and trivial as a representation of~$G(\bbA_\rmf)$, the functor~$\varpi(\frake_{\sfd,0}) \otimes \,\cdot\,$ is exact. We obtain a short exact sequence
\begin{gather*}
  \varpi(\frake_{\sfd,0}) \otimes \td\varpi_0
\cong
  \varpi(\frake_{\sfd,0}) \cdot \td\varpi_0
\lhra
  \varpi(\frake_{\sfd,0}) \otimes \td\varpi
\cong
  \varpi(\frake_{\sfd,0}) \cdot \td\varpi
\lthra
  \varpi(\frake_{\sfd,0}) \otimes \varpi
\tx{.}
\end{gather*}
Since the first two representation are contained in~$\cA(G(\QQ) \backslash G(\bbA), \sym^\sfd)$, we finish the proof.
\end{proof}

The prime example of Theorem~\ref{thm:automorphic_tensor_products} in this paper is given in the next statement.
\begin{theorem}
\label{thm:automorphic_tensor_products_E2}
We have inclusions of internal direct sums
\begin{gather*}
  \varpi(\frake_{1,0}) \oplus \varpi(\Evec_2)
\;\subset\;
  \varpi(\frake_{1,0}) \cdot \varpi(E_2)
\;\subset\;
  \cA\big( G(\QQ) \backslash G(\bbA), \sym^1 \big)
\end{gather*}
and a short exact sequence
\begin{gather*}
  \varpi(\frake_{1,0}) \oplus \varpi(\Evec_2)
\lhra
  \varpi(\frake_{1,0}) \otimes \varpi(E_2)
\lthra
  \varpi
\end{gather*}
for a vector-valued automorphic representation
\begin{gather*}
  \varpi
\;\cong\;
  \cD(2)
  \,\otimes\,
  \sideset{}{'}{\bigotimes}_{v \ne \infty}\,
  \Pi_{\frac{1}{2},v}
\tx{.}
\end{gather*}
\end{theorem}

\begin{corollary}%
\label{cor:E2vec_automophic_representation}
We have the restricted tensor product decomposition
\begin{gather*}
  \varpi(\Evec_2)
\;\cong\;
  \cD(0)
  \,\otimes\,
  \sideset{}{'}{\bigotimes}_{v \ne \infty}\,
  \Pi_{\frac{1}{2},v}
\tx{.}
\end{gather*}
\end{corollary}

\begin{proof}[Proof of Theorem~\ref{thm:automorphic_tensor_products_E2} and Corollary~\ref{cor:E2vec_automophic_representation}]
Observe that the constant function~$\rmL_2\,E_2$ lifts to a constant function in~$\varpi(E_2)$. That is, we have~$\varpi(1) \subset \varpi(E_2)$, and thus find that
\begin{gather*}
  \varpi(\frake_{1,0})
\;=\;
  \varpi(\frake_{1,0}) \cdot \varpi(1)
\;\subset\;
  \varpi(\frake_{1,0}) \cdot \varpi(E_2)
\tx{.}
\end{gather*}
In a similar vein, we conclude from~\eqref{eq:E2vec_LR_basis} that
\begin{gather}
\label{eq:E2vec_LR_basis_automorphic}
  \tdEvec_2
=
  \td\frake_{1,0} \td{E}_2 - \td\frake_{0,1} \mfrac{3}{\pi}
\;\in\;
  \varpi(\frake_{1,0}) \cdot \varpi(E_2)
\tx{.}
\end{gather}
To see that the sum representation~$\varpi(\frake_{1,0}) + \varpi(\Evec_2)$ is a direct sum, we inspect the summands as Harish-Chandra modules that are isotypical for~$\sym^1$ and~$\cD(0)$. Since these have different Harish-Chandra parameters, we have
\begin{gather*}
  \Ext_{(\frakg,K_\infty)}\big(\sym^1, \cD(0)\big)
=
  \Ext_{(\frakg,K_\infty)}\big(\cD(0), \sym^1\big)
\;=\;
  0
\tx{.}
\end{gather*}
This shows the first part of Theorem~\ref{thm:automorphic_tensor_products_E2}:
\begin{gather*}
  \varpi(\frake_{1,0}) + \varpi(\Evec_2)
\;=\;
  \varpi(\frake_{1,0}) \oplus \varpi(\Evec_2)
\;\subset\;
  \varpi(\frake_{1,0}) \cdot \varpi(E_2)
\tx{.}
\end{gather*}

From the first part of Theorem~\ref{thm:automorphic_tensor_products} and the short exact sequence~\eqref{eq:automorphic_form_E2_short_exact_sequence} for~$\varpi(E_2)$, we obtain the exact sequence
\begin{gather*}
  \varpi(\frake_{1,0})
\lhra
  \varpi(\frake_{1,0}) \cdot \varpi(E_2)
\lthra
  \varpi(\frake_{1,0}) \,\otimes\,
  \Big(
  \cD(1)
  \,\otimes\,
  \sideset{}{'}{\bigotimes}_{v \ne \infty}\,
  \Pi_{\frac{1}{2},v}
  \Big)
\tx{.}
\end{gather*}
Theorem~\ref{thm:finite_vector_valued_automorphic_representation} allows us to determine the right module. It is isomorphic to
\begin{gather*}
  \big( \cD(0) \oplus \cD(2) \big)
  \,\otimes\,
  \sideset{}{'}{\bigotimes}_{v \ne \infty}\,
  \Pi_{\frac{1}{2},v}
\tx{.}
\end{gather*}
We thus finish the proof when we establish Corollary~\ref{cor:E2vec_automophic_representation}.

To this end we inspect the~$K_\infty \times G(\bbA_f)$-module given by the lowest~$K_\infty$-type. We write~$\pi_k$ for the irreducible~$K_\infty$ representation corresponding to modular weight~$k$, and denote corresponding~$K_\infty$-isotypical components by square brackets~$\varpi[k]$. Further, we write~$\varpi_\rmf(E_2)$ and~$\varpi_\rmf(1)$ for the~$K_\infty \times G(\bbA_\rmf)$-modules generated by~$\wtd{E}_2$ and the constants. As a consequence of~\eqref{eq:automorphic_form_E2_short_exact_sequence}, we find that
\begin{gather*}
  \varpi(E_2)[0]
=
  \varpi(1)[0]
\;\cong\;
  \bbone
\quad\tx{and}\quad
  \varpi(E_2)[2]
\;\cong\;
  \pi_2
  \,\otimes\,
  \sideset{}{'}{\bigotimes}_{v \ne \infty}\,
  \Pi_{\frac{1}{2},v}
\tx{.}
\end{gather*}
Theorem~\ref{thm:finite_vector_valued_automorphic_representation} then shows that
\begin{gather*}
  \big( \varpi(\frake_{1,0}) \cdot \varpi(E_2) \big)[1]
\;=\;
  \wtd\frake_{0,1} \cdot \varpi(E_2)[0]
  \,\oplus\,
  \wtd\frake_{1,0} \cdot \varpi(E_2)[2]
\;\cong\;
  \pi_1
  \,\otimes\,
  \Big(
  \bbone_\rmf
  \,\oplus\,
  \sideset{}{'}{\bigotimes}_{v \ne \infty}\,
  \Pi_{\frac{1}{2},v}
  \Big)
\tx{,}
\end{gather*}
where~$\bbone_\rmf$ is the trivial~$G(\bbAf)$-representation.

From~\eqref{eq:E2vec_LR_basis_automorphic} and Theorem~\ref{thm:finite_vector_valued_automorphic_representation}, we conclude that
\begin{gather*}
  \varpi_\rmf(\Evec_2) + \varpi_\rmf(\frake_{0,1})
\;=\;
  \wtd\frake_{0,1} \cdot \varpi_\rmf(1)
  \,\oplus\,
  \wtd\frake_{1,0} \cdot \varpi_\rmf(E_2)
\tx{,}	
\end{gather*}
which we have identified in the previous equation. Taking the quotient by~$\wtd\frake_{0,1} \varpi_\rmf(1) \cong \pi_1 \otimes \bbone_\rmf$, we obtain Corollary~\ref{cor:E2vec_automophic_representation}.
\end{proof}

The sequence in Theorem~\ref{thm:automorphic_tensor_products_E2} does not split. In our final corollary, we identify the complement to~$\varpi(\Evec_2)$. It matches one of the Examples of case III(b) in Bringmann--Kudla's classification~\cite{bringmann-kudla-2018}.
\begin{corollary}%
\label{cor:BK_automophic_representation}
For
\begin{gather}
\label{eq:automorphic_tensor_products_E2:E2_indecomposable_analogue}
  f
=
  \frake_{1,0} E_2 + \mfrac{1}{2} \frake_{0,1} \rmR_2 E_2	
\end{gather}
we have an exact sequence
\begin{gather*}
  \varpi(\frake_{1,0})
\lhra
  \varpi(f)
\lthra
  \cD(2)
  \,\otimes\,
  \sideset{}{'}{\bigotimes}_{v \ne \infty}\,
  \Pi_{\frac{1}{2},v}
\tx{.}
\end{gather*}
\end{corollary}

\begin{remark}
As a special case of the calculation by Bringmann--Kudla in~(6.10) of~\cite{bringmann-kudla-2018}, one can use the relation~$\rmL \rmR\, E_2 = 2 E_2$ to directly verify that
\begin{gather*}
  \rmL_3 \big(
  \frake_{1,0} E_2 + \mfrac{1}{2} \frake_{0,1} \rmR_2 E_2	
  \big)
=
  \frake_{0,1} E_2 +
  \frake_{1,0} \mfrac{3}{\pi}
+
  \mfrac{1}{2} \frake_{0,1} \rmL_4 \rmR_2\, E_2	
=
  \frake_{1,0} \mfrac{3}{\pi}
\tx{.}
\end{gather*}
\end{remark}

\begin{proof}
It suffices to show that~$\varpi(f) + \varpi(\Evec_2) = \varpi(\frake_{1,0}) \cdot \varpi(E_2)$. Write~$\varpi$ for the left hand side. From the remark, we have~$\wtd\frake_{j,1-j} \in \varpi$ for~$j \in \{0,1\}$. Then the expression for~$\Evec_2$ in~\eqref{eq:E2vec_LR_basis} shows that~$\wtd\frake_{1,0} \wtd{E}_2 \in \varpi$. The action of~$\frakg$ implies that
\begin{gather*}
  \wtd\frake_{0,1} \wtd{E}_2 + \wtd\frake_{1,0} \wtd{\rmR\,E}_2 \in \varpi
\tx{,}
\end{gather*}
which together with the expression for~$f$ finishes the proof.
\end{proof}


\ifbool{nobiblatex}{%
  \bibliographystyle{alpha}%
  \bibliography{bibliography.bib}%
}{%
  \Needspace*{4em}
  \printbibliography[heading=none]
}


\Needspace*{3\baselineskip}
\noindent
\rule{\textwidth}{0.15em}

{\noindent\small
Chalmers tekniska högskola och G\"oteborgs Universitet,
Institutionen f\"or Matematiska vetenskaper,
SE-412 96 G\"oteborg, Sweden\\
E-mail: \url{martin@raum-brothers.eu}\\
Homepage: \url{http://martin.raum-brothers.eu}
}


\ifdraft{%
\listoftodos%
}

\end{document}
